\documentclass[]{interact}

\usepackage{epstopdf}

\usepackage[doublespacing]{setspace}

\usepackage{geometry}                
\usepackage{graphicx}
\usepackage{amssymb}
\usepackage{epstopdf}
\usepackage{amscd}
\usepackage{amsthm}
\usepackage{siunitx}
\usepackage{amsfonts}
\usepackage{amsmath}
\usepackage{latexsym}

\usepackage{epsfig}
\usepackage{graphicx, url}
\usepackage{cite}
\usepackage{color}
\usepackage{setspace}
\usepackage{caption}
\usepackage{subcaption}
\newtheorem{assumption}{Assumption}

\newcommand{\R}{\ensuremath{\mathbb{R}}}
\newcommand{\E}{\ensuremath{\mathbb{E}}}

\def\e{{\mathrm{e}}}

\usepackage[numbers,sort&compress]{natbib}
\bibpunct[, ]{[}{]}{,}{n}{,}{,}
\renewcommand\bibfont{\fontsize{10}{12}\selectfont}

\theoremstyle{plain}
\newtheorem{theorem}{Theorem}[section]
\newtheorem{lemma}[theorem]{Lemma}

\theoremstyle{definition}
\newtheorem{defin}[theorem]{Definition}
\newtheorem{example}[theorem]{Example}

\theoremstyle{remark}

\begin{document}


\title{Self-exciting jump processes and their asymptotic behaviour}

\author{
\name{K.~R. Dahl\textsuperscript{a}\thanks{CONTACT K.~R. Dahl. Email: kristrd@math.uio.no} and H. Eyjolfsson \textsuperscript{b}}
\affil{\textsuperscript{a} Department of Mathematics, University of Oslo, Norway; \textsuperscript{b}Department of Engineering, Reykjavik University, Iceland. heidare@ru.is.}
}

\maketitle

\begin{abstract}
The purpose of this paper is to investigate properties of self-exciting jump processes where the intensity is given by an SDE, which is driven by a finite activity stochastic jump process. The value of the intensity process immediately before a jump may influence the jump size distribution. We focus on properties of this intensity function, and show that the scaling limit of the intensity process equals the strong solution of the square-root diffusion process (Cox–Ingersoll–Ross process) in distribution. As a particular example, we study the case of a linear intensity process and derive explicit expressions for the expectation and variance in this case. Based on this, we show that once an appropriate scaling limit is taken, the resulting process exhibits infinite activity behaviour in distribution.
\end{abstract}



\section{Introduction}

Self-exciting processes were first studied by Hawkes \cite{Hawkes1}. Initially, the main application was in seismology; the modelling of earthquakes and their aftershocks. Over the last decade various versions of self-exciting processes have been used in financial applications (see e.g. Bacry et al. \cite{Bacry} and Embrechts et al. \cite{Embrechts}), to model group behaviour in social media (see Rizoiu et al. \cite{R}), and for predicting crime and terrorist acts (see Mohler \cite{Mohler} and Lewis et al. \cite{Lewis}). The class of self-exciting stochastic process models has the advantage of being very versatile in applications

In this paper, we investigate properties of self-exciting jump processes. Our definition of self-exciting processes, follows Eyjolfsson and Tj\o stheim \cite{EyjolfssonTjostheim}. The self-exciting process is essentially a counting process, which counts the number of occurred shocks at any given time. The intensity process of the self-exciting process, denoted by $\lambda(t)$, determines the probability of shocks occurring in the infinitesimal interval $(t,t+dt)$ conditioned on the information at time $t$. We assume that this intensity process has Markovian stochastic differential equation (SDE) dynamics. The self-exciting processes considered in this paper differ from Hawkes processes, see Hawkes \cite{Hawkes1} and Hawkes and Oakes \cite{Hawkes2}, because of the stochastic jump size of the self-exciting process may depend on the current value of the intensity process. As noted in Eyjolfsson and Tj\o stheim \cite{EyjolfssonTjostheim}, this kind of self-exciting process is a generalisation of the exponential Hawkes model.


The self-exciting model we present in this paper is a finite activity jump process (like the compound Poisson process) in the sense that in bounded time intervals it only produces a finite amount of jumps. However, as we will discuss in Section \ref{sec: finite_infinite}, when model parameters are chosen in a suitable way, and then passed to a limit one obtains an infinite activity process in the limit which can be thought of as an infinite activity analogue of the self-exciting process. This is similar to how the gamma and inverse Gaussian processes are infinite activity limits of compound Poisson processes. We emphasize that the class of self-exciting processes has the ability to produce periods (clusters) of high activity between periods of low activity, which is something that L\'evy processes can not reproduce.

Jaisson and Rosenbaum \cite{JaissonRosenbaum} derive limit theorems for Hawkes processes which are  nearly unstable. These processes are such that the $\mathcal{L}^1$-norm of their kernel is close to unity. Jaisson and Rosenbaum \cite{JaissonRosenbaum} show that after a rescaling, the nearly unstable Hawkes counting processes asymptotically behave like integrated Cox–Ingersoll–Ross models, and the nearly unstable intensity processes asymptotically behave like Cox–Ingersoll–Ross models. In the same spirit, we prove limit theorems for the self-exciting processes. In Theorem \ref{thm:lam_convergence}, we show that for a fixed time, $t \geq 0$, the scaling limit of the intensity process of the self-exciting process at time $t$ behaves like the Cox–Ingersoll–Ross square-root process in distribution. We moreover prove a similar result for the integrated scaled intensity process, which converges to the integrated Cox–Ingersoll–Ross square-root process in distribution. We reiterate that our self-exciting processes differ from the Hawkes processes in that our processes are specified by a Markovian SDE with stochastic jumps, whereas the Hawkes intensity process has constant jumps and specified by a kernel function with  $\mathcal{L}^1$-norm less than one. 

The paper is structured as follows: In Section \ref{sec: self-exciting} we introduce the framework for self-exciting stochastic processes and illustrate such processes via a numerical example. In Section \ref{sec: finite_infinite}, we show that self-exciting processes can exhibit both finite and infinite activity type behaviour. We also prove that the scaling limit of the intensity process equals the square-root process in distribution. In Section \ref{sec: linear}, we study a particular case where the intensity process of the self-exciting process is assumed to be linear. We derive the expected value and variance of this linear intensity process, as well as the moments of the integrated intensity. Finally, in Section \ref{sec:Conclusion}, we conclude and sketch ideas on further research.

\section{Self-exciting stochastic jump processes}
\label{sec: self-exciting}

Our definition of self-exciting processes, follows Eyjolfsson and Tj\o stheim \cite{EyjolfssonTjostheim}. Essentially, the self-exciting process is a counting process, which counts the number of shocks which have occurred at any given time. Let $(\Omega,\mathcal{F})$ denote a measurable space, and let $\{T_n\}_{n \geq 1}$ be a point process taking values in $\R_+$. The sequence $\{T_n\}_{n \geq 1}$ is assumed non-negative and non-decreasing, i.e. $0 \leq T_1 \leq T_2 \leq \cdots$ holds. The sequence represents times of successive events. The counting process, $N(t)$, associated to the point process,
\begin{equation}\label{def:N}
N(t) := \sum_{n \geq 1} 1_{\{T_n \leq t\}},
\end{equation}
where $t \geq 0$, is the counting process which records all the jumps of the point process. The rate at which the events occur is furthermore dictated by the intensity process, which we define in what follows. We identify a point process with its counting process \eqref{def:N} and let
\begin{equation*}
\mathcal{F}_t^N := \sigma\{N(s) : 0 \leq s \leq t \},
\end{equation*}
where $t \geq 0$. That is, $\{\mathcal{F}_t^N\}_{t \geq 0}$ is the filtration generated by the counting process. Assume that we are given a point process adapted to some filtration $\{\mathcal{F}_t\}_{t \geq 0}$, with $\mathcal{F}_t^N \subset \mathcal{F}_t$ for all $t \geq 0$. Suppose that $N(t)$ admits a c\`adl\`ag  $\{\mathcal{F}_t\}_{t \geq 0}$-adapted, and thus predictable, intensity $\lambda(t)$, such that  
$$
\E\left[\int_0^\infty  f(s) dN(s) \right] = \E\left[\int_0^\infty f(s) \lambda(s)ds \right],
$$
holds for all predictable $f: \Omega \times \R_+ \to [-\infty,\infty]$. Note that this means the the process $t \mapsto N(t) - \lambda(t)$ is a martingale, and that the intensity process $\lambda(t)$ determines the probability of shocks occurring in the infinitesimal interval $(t,t+dt)$ conditioned on $\mathcal{F}_t$. Note in particular that if the intensity is constant, $\lambda(t) = \lambda_0 > 0$, holds  for all $t \geq 0$, then $N(t)$ is a standard homogeneous Poisson process with intensity $\lambda_0$.

We assume that the intensity process admits Markovian SDE dynamics. Each jump has a particular size, which feeds into (i.e. excites) the intensity, and typically raises the intensity immediately after the shock has occurred, although the intensity will then revert back to some mean level in the absence of further shocks. The size of the shock can furthermore influence how much the likelihood of of further shocks is increased (i.e. the level of excitement). Hence, a particularly large shock may for example lead to a high likelihood of aftershocks, whereas a small shock may be less likely to excite the intensity and thus cause further aftershocks. Thus, the model class allows the shocks to vary in size, and the size of each shock determines the level of the corresponding intensity process excitation. If the intensity becomes high enough, a cluster of shocks might appear. 

Consider the stochastic jump process, $U(t)$, given by
\begin{equation}\label{def:U}
U(t) = \sum_{k=1}^{N(t)} X_{k},
\end{equation}
where $\{N(t)\}_{t \geq 0}$ is the counting process \eqref{def:N}, and $\{X_k\}_{k \in \mathbb{N}}$ is a family of random variables, $X_k$ has the probability distribution $\nu(\lambda(T_k-),\cdot)$, for a given family $\{\nu(\lambda,\cdot)\}_{\lambda > 0}$ of probability distributions, and $t- := \lim_{s \uparrow t} s$. Thus we allow the value of the intensity process immediately before the jump to influence the jump size distribution. We introduce the stochastic differential equation (SDE)

\begin{equation}\label{def:sde}
d\lambda(t) = \mu(\lambda(t))dt +  \beta dU(t),
\end{equation}

\noindent $\lambda(0) = \lambda_0$, where $\beta \in \R$ is a constant and we assume that $\mu:\R_+ \to \R$, is Lipschitz continuous. 

\begin{defin}\label{def: self-exciting}
An SDE-driven self-exciting jump process is a stochastic jump process \eqref{def:U} with the intensity $\lambda(t)$, given by the SDE \eqref{def:sde}, with jump-sizes, $X_k$, which follow the probability distribution $\nu(\lambda(T_k-),\cdot)$. Here, $\{\nu(\lambda,\cdot)\}_{\lambda > 0}$ is a family of probability distributions, and $\nu(\lambda,\cdot)$ is supported on $[\lambda_0-\lambda,\infty)$.
\end{defin}
From the above Definition \ref{def: self-exciting}, we see that the jumps feed into the intensity via the jump process $U(t)$. Furthermore the value of the intenisty process immediately before the jump is a parameter in the the jump-size probability distribution. This means that the intensity level prior to a jump can determine the size of the next jump. To prevent an explosion happening in finite time the function $\mu$ must be negative for high values of $\lambda(t)$.

Also note that the self-exciting processes in Definition \ref{def: self-exciting} differ from Hawkes processes, see Hawkes \cite{Hawkes1} and Hawkes and Oakes \cite{Hawkes2}, because of the stochastic jump size modelled via the family $\{X_k\}_{k \in \mathbb{N}}$ of random variables which may depend on the current value of the intensity process. Actually, this kind of self-exciting process generalises the exponential Hawkes model (i.e., the exponential Hawkes process is a special case of Definition \ref{def: self-exciting}), see Eyjolfsson and Tj\o stheim \cite{EyjolfssonTjostheim} for more details.


\medskip

\begin{example} (Simulation of a self-exciting process $U(t)$) 
\label{ex:simulation_self-exciting}

To illustrate, we simulate two paths of the same self-exciting process and plot the intensity $\lambda(t)$ as well as the corresponding self-exciting process $U(t)$. Following Eyjolfsson and Tj{\o}stheim \cite{EyjolfssonTjostheim}, we consider a non-linear intensity process $\lambda(t)$ given as the solution to the following SDE

\begin{equation}
\label{eq: nonlinear_intensity}
\begin{array}{llll}
d \lambda(t) &=& (\alpha + \delta \exp(-\gamma \lambda(t)^2)) (\lambda_0 - \lambda(t))dt + \beta dU(t), \\[\smallskipamount]
\lambda(0) &=& \lambda_0.
\end{array}
\end{equation}

As mentioned in Eyjolfsson and Tj{\o}stheim \cite{EyjolfssonTjostheim}, the speed of mean reversion in the SDE \eqref{eq: nonlinear_intensity} varies between $\alpha + \delta \exp(-\gamma \lambda_0^2)$ for $\lambda(t)=\lambda_0$ and decreases towards $\alpha$ when $\lambda(t)$ increases. The interpretation of this is that in low activity periods, the effect of a jump fades out faster than in high activity periods.

For the simulation, we choose $\lambda_0 = 0.05$, $\alpha=0.1233$, $\beta=0.0399$ and the jumps are simulated from an inverse Gaussian distribution with parameters 1.9389 (mean) and 5.4943 (shape). The parameter values were chosen based on Eyjolfsson and Tj{\o}stheim \cite{EyjolfssonTjostheim}, but the choice of $\lambda_0$ was modified slightly to better display the particular structure of the self-exciting process $U(t)$. The simulation was performed using a thinning algorithm from Ogata \cite{Ogata}.

In Figures \ref{fig:selfexciting_1} and \ref{fig:selfexciting_2} we have plotted two different paths of the self-exciting process $U(t)$ with the corresponding intensity process $\lambda(t)$. Periods with a lot of jump activity in the intensity process correspond to a large increases in the self-exciting process. Periods where there are no jumps in the intensity process correspond to plateaus (i.e., no change) in the corresponding self-exciting process. 

\begin{figure}[h!]
\centering
\begin{subfigure}{.5\textwidth}
  \centering
  \includegraphics[width=\linewidth]{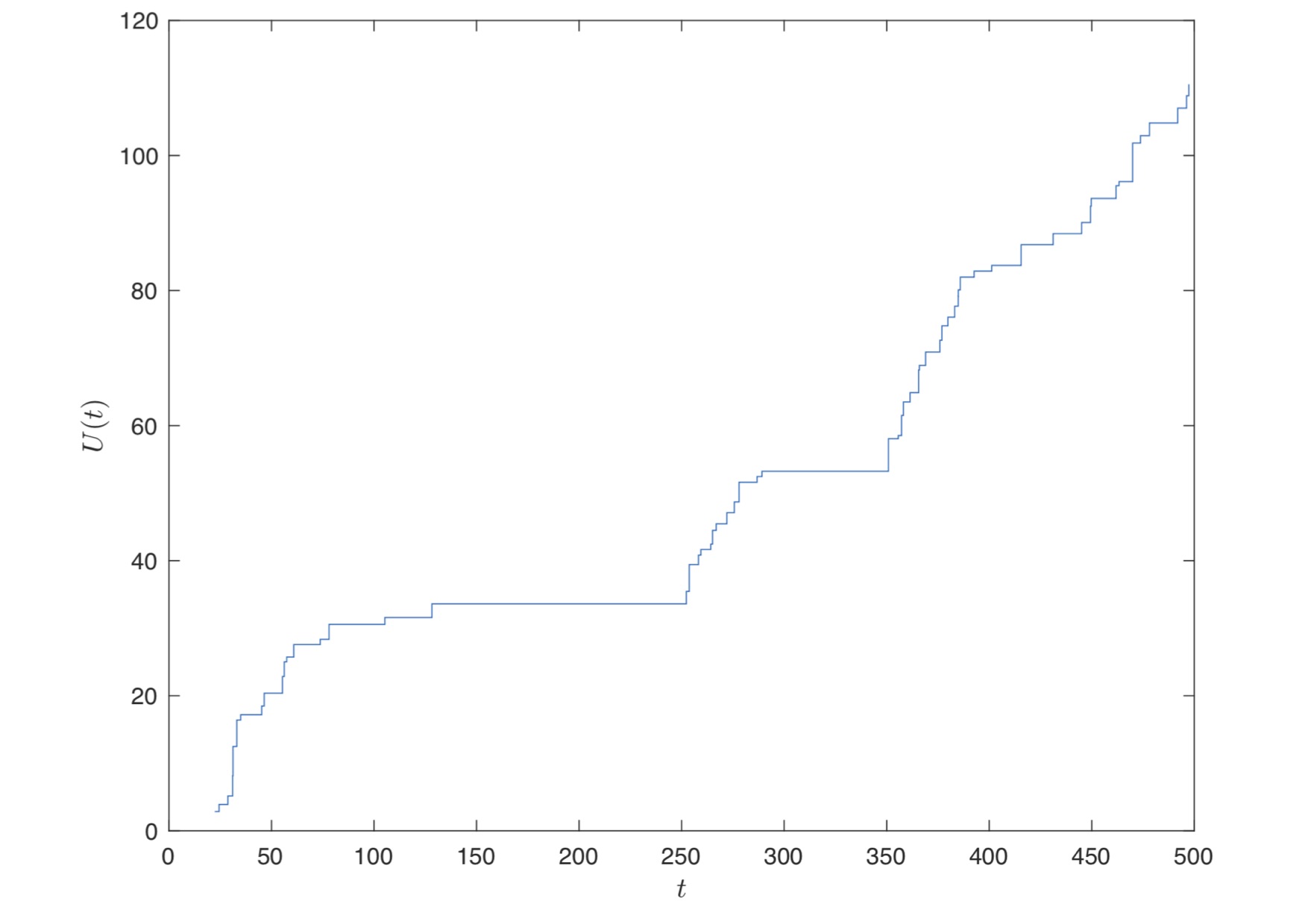}
\end{subfigure}%
\begin{subfigure}{.5\textwidth}
  \centering
  \includegraphics[width=\linewidth]{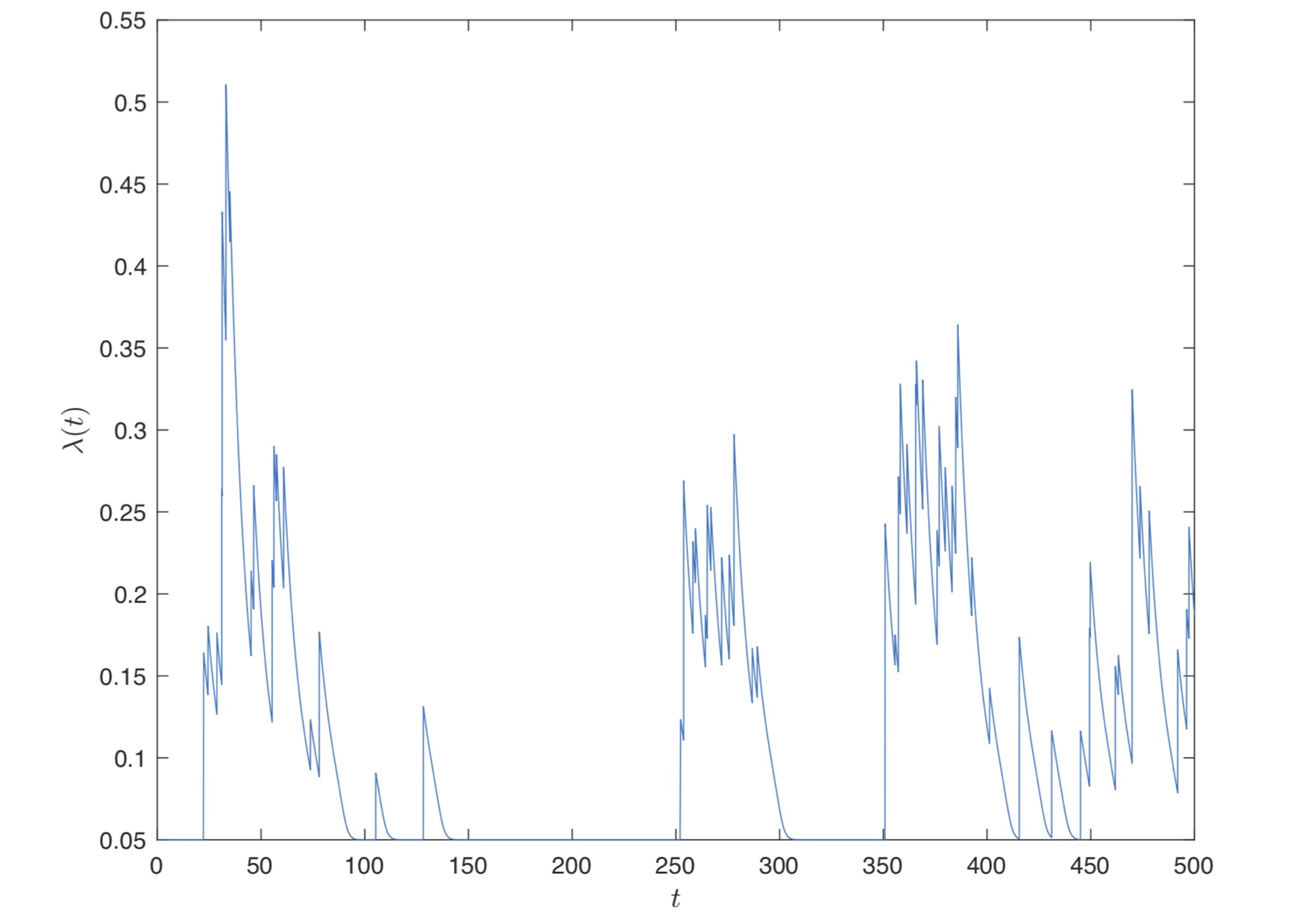}
\end{subfigure}
  \caption{A path of the self-exciting process $U(t)$ with the corresponding intensity $\lambda(t)$.}
  \label{fig:selfexciting_1}
\end{figure}

\begin{figure}[h!]
\centering
\begin{subfigure}{.5\textwidth}
  \centering
  \includegraphics[width=\linewidth]{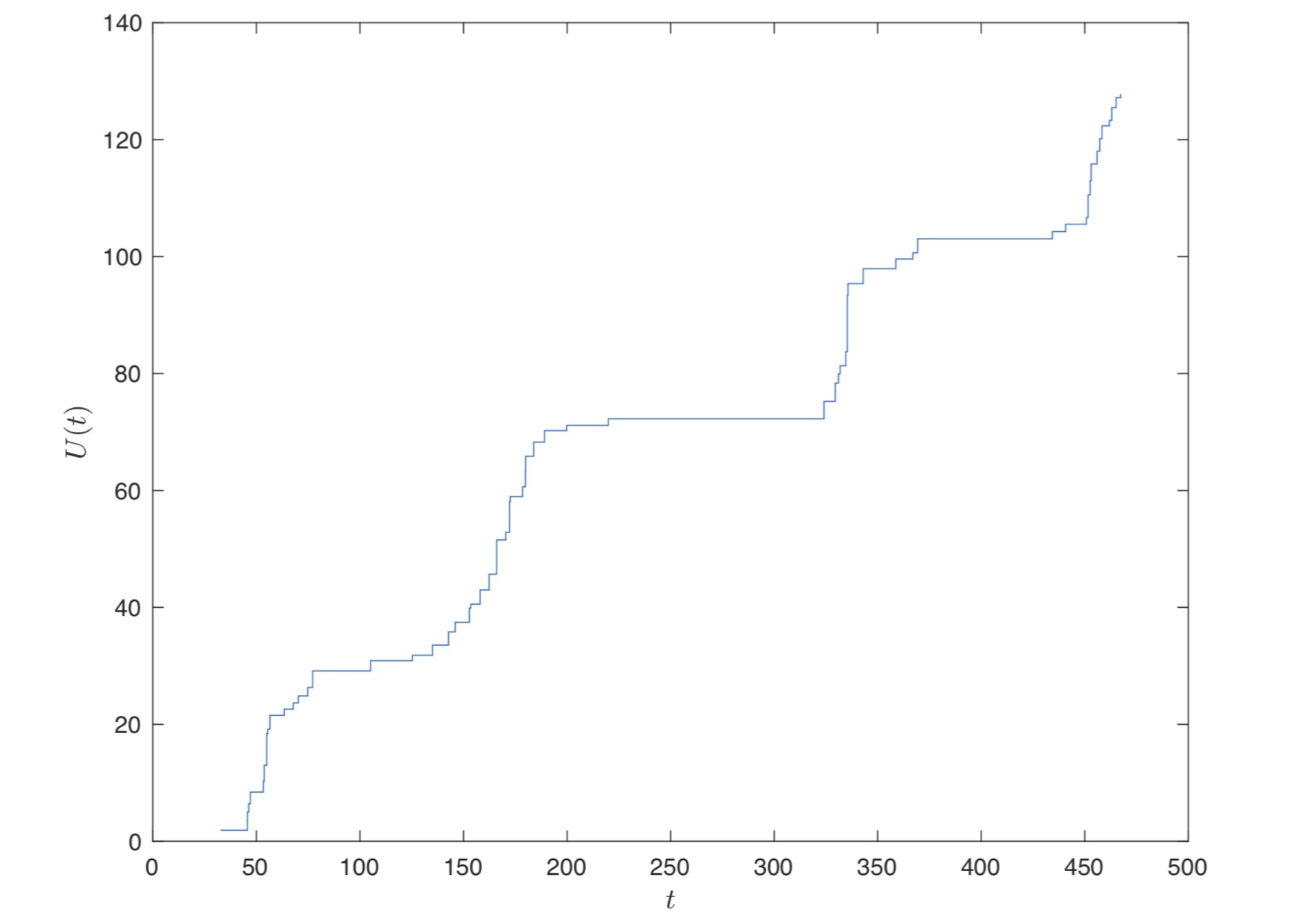}
\end{subfigure}%
\begin{subfigure}{.5\textwidth}
  \centering
  \includegraphics[width=\linewidth]{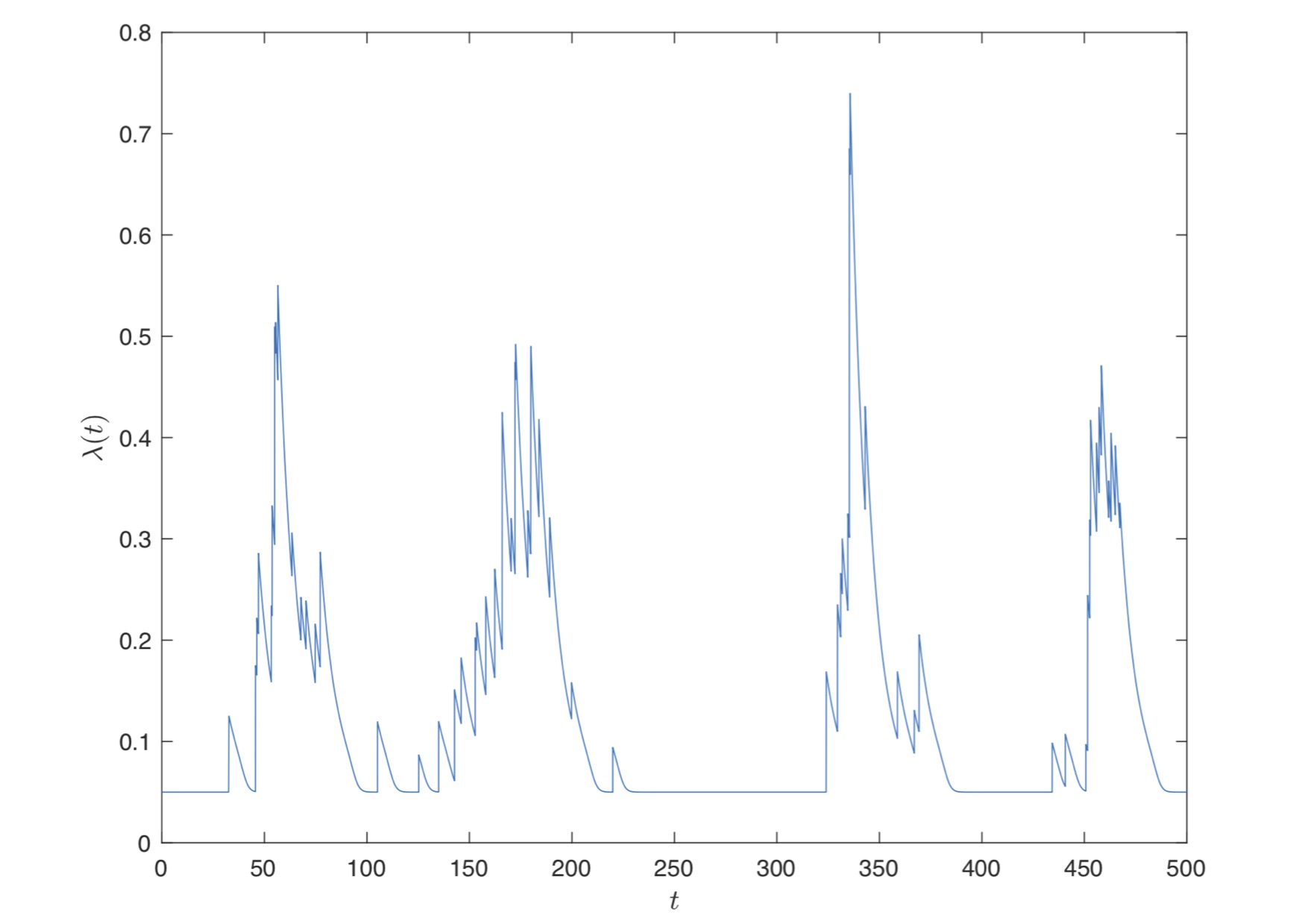}
\end{subfigure}
  \caption{Another path of the self-exciting process $U(t)$ with the corresponding intensity $\lambda(t)$.}
  \label{fig:selfexciting_2}
\end{figure}

\end{example}





\section{Finite and infinite activity of self-exciting processes} 
\label{sec: finite_infinite}


A stochastic jump process is said to have \emph{finite activity} if it has finitely many jumps in finite time. In contrast, an \emph{infinite activity} jump process can have an infinite number of jumps in finite time. As an example, compound Poisson processes have finite activity, while gamma processes and generalized inverse Gaussian processes have infinite activity and jumps that are infinitesimally small (to avoid explosion of the processes). 

A natural question is whether the self-exciting process defined in Section \ref{sec: self-exciting} is of finite or infinite activity. In Eyjolfsson and Tj{\o}stheim \cite{EyjolfssonTjostheim}, conditions are provided which ensure that the counting process $N(t)$ associated with the intensity SDE \eqref{def:sde} does not explode in finite time. This means that the self-exciting process has finitely many jumps on any compact interval, and is thus of finite activity as we have described it above. Let 
\[
\Lambda(t) := \int_0^t \lambda(s) ds.
\]
Then, a condition for the self-exciting process to be of finite activity is:

\begin{lemma}
\label{lemma: finite_activity}
Assume that for all times $T >0$, 
\begin{equation}
\label{eq: lambda}
\Lambda(T) < \infty \mbox{ holds almost surely.}
\end{equation}
Then, the SDE-driven self-exciting process (as defined in Section \ref{sec: self-exciting}) is of finite activity. 
\end{lemma}

\begin{proof}
The Lemma follows from Assumption 1, and the comments thereafter, in Eyjolfsson and Tj{\o}stheim \cite{EyjolfssonTjostheim} combined with the definition of finite activity above.
\end{proof}

If the intensity process $\lambda(t)$ in equation \eqref{def:sde} does not satisfy assumption \eqref{eq: lambda} of Lemma \ref{lemma: finite_activity}, then the counting process $N(t)$ may explode in finite time. In this case, the self-exciting process may have infinite activity. By choosing a jump size distribution, $\nu$, with positive jumps, and an appropriate drift term $\mu$ which does not drive the intensity sufficiently fast downwards between jumps in the SDE \eqref{def:sde}, we can ensure an explosion of the number of jumps $N(t)$ in finite time. However, for such a model to be meaningful we must ensure simultaneously that the jump-sizes, $X_k$, are infinitesimal like they are in infinite activity L\'evy processes. Otherwise, the jump process itself will explode. In the current section we show that by taking an appropriate scaling limit the resulting limit of an SDE-driven self-exciting processes is an infinite activity process in distribution. 

\subsection{Infinite activity self-exciting processes}
\label{sec: inf activity}
In this section we discuss infinite activity self-exciting processes and how they can be obtained as limits of finite activity self-exciting processes. If the parameters which govern the dynamics of the SDE intensity are changed in a way such that $\Lambda(t)$ increases, while the jump-sizes, $X_1,X_2,\ldots$, become smaller and smaller simultaneously, then, as we pass to the limit, we can ensure that in the limit the self-exciting jump process, $U(t)$, has infinite activity. 

Essentially there are two non-trivial possibilities. The first one is that an appropriately scaled version of the stochastic intensity, converges in law to a non-random positive measure. The second option of interest is that the scaled stochastic intensity converges in law to a stochastic process. In what follows we study conditions under which the scaled stochastic intensity process converges in law to a stochastic process.

We study the asymptotic properties of the intensity process as we increase the jump intensity and decrease the jump size distribution simultaneously. To that end, consider a self-exciting process which depends on a parameter, $k \geq 1$, i.e. a process with an intensity,
\begin{equation}\label{lam_k}
d\lambda_k(t) = \mu_k(\lambda_k(s))dt + \beta_kdU_k(t), 
\end{equation}
where $\lambda_k(0) = \lambda_{0k}$ and iid non-negative jump size distribution with (joint) cumulative probability distribution function $F_k$. Suppose furthermore that $\{a_k\}$, $k \geq 1$ is a sequence such that $a_k > 0$ for all $k \geq 1$ and $\lim_{k \to \infty} a_k = 0$. We study the behaviour of the scaled intensity process 
\begin{equation}\label{lam_hat_k}
\hat \lambda_k(t) := a_k\lambda_k(t).
\end{equation}
Now let $m_{k,n}(t) := E[\lambda_k^n(t)]$ and $\hat m_{k,n}(t) := E[\hat \lambda_k^n(t)]$.  
In what follows we shall make the following assumption.

\begin{assumption}\label{ass:stability}
Given the intensity processes \eqref{lam_k} and the corresponding scaled processes \eqref{lam_hat_k}, suppose that $\lim_{k \to \infty} a_k \lambda_{0k} = 0$, there is a sequence of non-negative functions $\{g_k\}$ such that $a_kg_k(\lambda)$ is bounded for each $k \geq 1$ and $\lim_{k \to \infty} a_kg_k(\lambda) = c_0 \in \R$, where the convergence is uniform in $\lambda$, and it holds that 
$$
\mu_k(\lambda) + \beta_k E[X_k]\lambda = g_{k}(\lambda) - c_1\lambda
$$
where $c_1 > 0$ and for each $k \geq 1$, the jump sizes are non-negative, independent and identically distributed, with finite moments. Suppose furthermore that  
$$
\lim_{k \to \infty} a_k^{j-1} \beta_k^{j}E[X_k^{j}] = \begin{cases}
c_2 & \text{if } j = 2\\  
0 & \text{if } j > 2,
\end{cases} 
$$
where the random variable $X_k$ represents a jump in $U_k(t)$.
\end{assumption}

According to Eyjolfsson and Tj\o stheim \cite{EyjolfssonTjostheim} the generator of the intensity process $\lambda(t)$ is given by 
$$
(\mathcal{A}f)(\lambda) = \mu(\lambda) f'(\lambda) + \lambda \int (f(\lambda + \beta x) - f(\lambda))\nu(\lambda,dx),
$$
whenever $f$ is in the domain of the generator. Furthermore, if $f$ is in the domain of the generator, then the Dynkin formula is verified, i.e.,
$$
E[f(\lambda(t))] = f(\lambda) + E[\int_0^t \mathcal{A}f(\lambda(r)) dr],
$$
when it it assumed that the initial value is $\lambda(0) = \lambda$. Letting $f(\lambda) = \lambda^n$, where $n \geq 1$, and employing the binomial theorem, we obtain that 
$$
(\mathcal{A}f)(\lambda) = n(\mu(\lambda) + \beta E[X]\lambda)\lambda^{n-1} + \sum_{j=0}^{n-2} {n \choose j} \beta^{n-j} E[X^{n-j}] \lambda^{j+1}.
$$
So, if $m_n(t) = E[\lambda^n(t)]$ denotes the $n$th moment of the intensity $\lambda(t)$,  an application of Dynkin's formula and Fubini yields
\begin{align} \label{def:moment}
\begin{split}
m_n(t) &= \lambda^n + \int_0^t nE[(\mu(\lambda(s)) + \beta E[X]\lambda(s))\lambda^{n-1}(s)]ds  \\
 \ \ &+ \sum_{j=0}^{n-2} {n \choose j} \beta^{n-j} E[X^{n-j}] \int_0^t m_{j+1}(s) ds, 
\end{split}
\end{align}
where the sum is dropped when $n=1$. 

In the following lemma, we show that the $n$th moment of the scaled intensity process, $\hat m_{k,n}(t) = E[\hat \lambda_k^n(t)]$, is bounded by an $n$th degree polynomial. Later, this result will be used to prove that the scaling limit of the intensity equals the square-root process in distribution.

\begin{lemma}\label{lem:finite_moments}
For each $n \geq 0$, the moment $\hat m_{k,n}(t)$ is bounded by an $n$th degree polynomial which is independent of $k \geq 1$.
\end{lemma}

\begin{proof}
By \eqref{def:moment} it holds that
\begin{align*}
\hat m_{k,n}(t) &= a_k^n\left( \lambda_{0k}^n + \int_0^t nE[(\mu_k(\lambda_k(s)) + \beta_k E[X_k]\lambda_k(s))\lambda_k^{n-1}(s)]ds \right. \\
 \ \  &+ \left.\sum_{j=0}^{n-2} {n \choose j} \beta_k^{n-j} E[X_k^{n-j}] \int_0^t m_{k,j+1}(s) ds \right).
\end{align*}
Notice that since an intensity is always non-negative it holds for any $n \geq 1$ that  
\begin{align*}
E[a_k^n (g_k(\lambda_k(s)) - c_1\lambda_k(s)\lambda_k^{n-1}(s)]  &\leq E[a_k g_k(\lambda_k(s))\hat \lambda_k^{n-1}(s)] \\
&\leq (\|a_kg_k - c_0\|_\infty + c_0)\hat m_{k,n-1}(s) ,
\end{align*}
where $\|f\|_\infty = \sup_x |f(x)|$ denotes the sup norm. Hence, by Assumption~\ref{ass:stability} it follows that $\hat m_{k,1}(t) \leq (\sup_k\|a_kg_k - c_0\|_\infty + c_0)t$, and by induction it follows that $\hat m_{k,n}(t)$ is bounded by a $n$th degree polynomial which is independent of $k$. 
\end{proof}

We now show that, given our assumption, for each $t \geq 0$ the scaling limit of the intensity process \eqref{lam_hat_k} as $k \to \infty$ in distribution is the square-root process which is given by the strong solution of the SDE 
\begin{equation}\label{def:CIR_model}
dY(t) = (c_0 - c_1Y(t))dt + \sqrt{c_2 Y(t)}dB(t),
\end{equation} 
where $Y(0) = 0$ and $B(t)$ denotes Brownian motion. 

\begin{theorem}\label{thm:lam_convergence}
For any $t \geq 0$, it holds that $\hat \lambda_k(t) \to Y(t)$, as $k \to \infty$ in distribution, where $Y(t)$ is given by \eqref{def:CIR_model}.
\end{theorem}

\begin{proof}
According to Assumption~\ref{ass:stability} and Lemma~\ref{lem:finite_moments} we may apply \eqref{def:moment} and the dominated convergence theorem to conclude that
\begin{align*}
\lim_{k \to \infty} \hat m_{k,n}(t) &= \lim_{k \to \infty} a_k^n\left( \lambda_{0k}^n + \int_0^t nE[(\mu_k(\lambda_k(s)) + \beta_k E[X_k]\lambda_k(s))\lambda_k^{n-1}(s)]ds \right. \\
 \ \  &+ \left.\sum_{j=0}^{n-2} {n \choose j} \beta_k^{n-j} E[X_k^{n-j}] \int_0^t m_{k,j+1}(s) ds \right) \\
&=  n \int_0^t E[ \lim_{k \to \infty} \left(  a_kg_k(\lambda_k(s)) -c_1 \hat \lambda_{k}(s) \right) \hat \lambda_k^{n-1}(s)]ds  \\
 \ \  &+  \frac{n(n-1)}{2} c_2 \int_0^t \lim_{k \to \infty}\hat m_{k,n-1}(s) ds.
\end{align*}

Hence, by an application of the fundamental theorem of the calculus, it follows that for each $n \geq 1$, the $n$th moment of the limit of the scaled intensity, $\lim_{k \to \infty} \hat \lambda_k(t)$, solves the ODE
\begin{align*}
y_n ' = -nc_1y_n + \left(nc_0 + \frac{n(n-1)}{2}c_2\right)y_{n-1},
\end{align*}
where $y_n(0) = 0$. 

Now, by applying  It\^o's lemma with $f(x) = x^n$ to the stochastic process $Y(t)$ it follows that the $n$th moment of $Y(t)$ also verifies the above ODE, for every $n \geq 1$. Moreover, Dufresne~\cite{Dufresne} shows that the series 
$$
\sum_{n=0}^\infty \frac{s^n}{n!} E[Y^n(t)]
$$
converges when $s$ is small enough, so that the moment generating function (MGF) of $Y(t)$ can be obtained in this way by evaluating the above series. Thus, we have shown that the MGF of $\lim_{k \to \infty} \hat \lambda_k(t)$ equals the MGF of $Y(t)$ in a neighbourhood around $0$, and hence (see e.g. section 30 in Billingsley~\cite{Billingsley}) they are equal in distribution.
\end{proof}

A similar result holds for the integrated intensity process.

\begin{theorem}\label{thm:int_lam_convergence}
For any $t \geq 0$, it holds that $\int_0^t \hat \lambda_k(s) ds \to \int_0^t Y(s)ds$, as $k \to \infty$ in distribution, where $Y(t)$ is given by \eqref{def:CIR_model}.
\end{theorem}

\begin{proof}
First we show that the moments of the integrated intensity, $\hat \Lambda_k(t) = \int_0^t \hat \lambda_k(s) ds$, are finite independently of the parameter $k \geq 1$. Note that 
$$
\hat \Lambda_k^n(t) = \left(\int_0^t \hat \lambda_k(s) ds\right)^n  \leq t^n \left( \sup_{s \in [0,t]} \hat \lambda_k(s) \right)^n.
$$
Since the jumps are non-negative, it follows that the intensity, $s \mapsto \hat \lambda_k(s)$ (and thus $s \mapsto \hat \lambda_k^n(s)$) is a.s. upper semi-continuous. A property of upper semi-continuous functions is that it attains its supremum on a compact set. A consequence of this is that for a fixed $\omega \in \Omega$ and $\epsilon > 0$, the set 
$$
A_\epsilon = \{s \in [0,t] : \hat \lambda_k^n(s) \geq \sup_{r \in [0,t]} \hat \lambda_k^n(r) - \epsilon \}
$$ 
is not empty and closed. Hence, since $\epsilon > 0$ is arbitrary, it must hold that $\sup_{s \in [0,t]}E[ \hat \lambda_k^n(s)] \geq E[ \sup_{s \in [0,t]} \hat \lambda_k^n(s)]$, and conversely it clearly holds that $\sup_{s \in [0,t]}E[ \hat \lambda_k^n(s)] \leq E[ \sup_{s \in [0,t]} \hat \lambda_k^n(s)]$. Similar arguments can moreover be employed to show that $(\sup_{s \in [0,t]} \hat \lambda_k(s))^n = \sup_{s \in [0,t]} \hat \lambda_k^n(s)$. It follows that 
\begin{align*}
E[\hat \Lambda_k^n(t)] \leq t^n E[\left( \sup_{s \in [0,t]} \hat \lambda_k(s) \right)^n] = t^n \sup_{s \in [0,t]} \hat m_{k,n}(s),
\end{align*}
so, according to Lemma~\ref{lem:finite_moments} it holds that $E[\hat \Lambda_k^n(t)] < C(t) < \infty$ where $C(t) > 0$ is independent of $k$.

Suppose that $m,n \geq 0$. According to It\^o's formula (see Theorem II.33 in Protter~\cite{Protter}) applied to the product $t \mapsto \hat \Lambda_k^m(t) \hat \lambda_k^n(t)$ it holds that 
\begin{align*}
&\hat \Lambda_k^m(t) \hat \lambda_k^n(t) = \int_0^t  \hat \lambda_k^n(s) d(\hat \Lambda_k^m(s)) + \int_0^t  \hat \Lambda_k^m(s) d(\hat \lambda_k^n(s)) \\
\ &+ \sum_{0 \leq s \leq t} \{ \hat \Lambda_k^m(s) \hat \lambda_k^n(s) - \hat \Lambda_k^m(s-) \hat \lambda_k^n(s-) - \hat \lambda_k^m(s-) \Delta \hat \Lambda_k^n(s) - \hat \Lambda_k^m(s-) \Delta \hat \lambda_k^n(s)\} \\
&= \int_0^t m \hat \Lambda_k^{m-1}(s) \hat \lambda_k^{n+1}(s) ds + \int_0^t  \hat \Lambda_k^m(s) n \hat \lambda_k^{n-1}(s) a_k \mu_k(\lambda_{k}(s))ds \\
\ &+ \sum_{0 \leq s \leq t} \hat \Lambda_k^m(s)\left(\hat \lambda_k^n(s) - \hat \lambda_k^n(s-)\right).
\end{align*}
Now, let $M_{k,m,n}(t):= E[\hat \Lambda_k^m(t) \hat \lambda_k^n(t)]$, for $m,n \geq 0$. Then, according to the Cauchy-Schwarz inequality it holds that 
\begin{equation*}
M_{k,m,n}(t) \leq \left( E[\hat \Lambda_k^{2m}(t)] \hat m_{k,2n}(t) \right)^{1/2}  < C(t),
\end{equation*}
where the constant $C(t) > 0$ is independent of $k$. Note moreover that, since $t \mapsto N_k(t) - \Lambda_k(t)$ is a martingale, it holds that
\begin{align*}
&E[ \sum_{0 \leq s \leq t} \hat \Lambda_k^m(s)\left(\hat \lambda_k^n(s) - \hat \lambda_k^n(s-)\right)] \\ 
\ &= E[ \int_0^t \hat \Lambda_k^m(s)\left( \int \left( \hat \lambda_k(s-) +  a_k\beta_k x \right)^n\nu_k(dx) - \hat \lambda_k^n(s-)\right)dN_k(s)] \\
&= E[ \int_0^t \hat \Lambda_k^m(s) \sum_{j=0}^{n-1} {n \choose j} \hat \lambda_k^j(s-) a_k^{n-j}\beta_k^{n-j} E[X_k^{n-j}]\lambda_{k}(s) ds] \\
&= \sum_{j=0}^{n-1} {n \choose j} a_k^{n-j-1}\beta_k^{n-j} E[X_k^{n-j}]\int_0^tE[\hat \Lambda_k^m(s)\hat \lambda_k^{j+1}(s) ] ds,
\end{align*}
where we have applied the binomial theorem and Fubini's theorem. Hence, we may apply the dominated convergence theorem, together with Assumption~\ref{ass:stability} to conclude that
\begin{align*}
&\lim_{k \to \infty} M_{k,m,n}(t) = \int_0^t m\lim_{k \to \infty}  M_{k,m-1,n+1}(s) ds + \int_0^t n c_0 \lim_{k \to \infty}  M_{k,m,n-1}(s) ds \\
& \ -  \int_0^t nc_1 \lim_{k \to \infty}  M_{k,m,n}(s) ds + \int_0^t c_2\frac{n(n-1)}{2}\lim_{k \to \infty} M_{k,m,n-1}(s)ds.
\end{align*}
An application of the fundamental theorem of the calculus thus yields that for each $m,n \geq 0$ the limit 
$\lim_{k \to \infty} M_{k,m,n}(t)$ solves the sytem of ODE's given by 
$$
y'_{m,n} = -nc_1 y_{m,n}  + (nc_0 + \frac{n(n-1)}{2}c_2)y_{m,n-1} + m y_{m-1,n+1}, 
$$
with $y_{m,n}(0)=0$. By applying It\^o's formula to $t \mapsto (\int_0^t Y(s)ds)^m Y^n(t)$ one can moreover show that the moments  $t \mapsto E[ (\int_0^t Y(s)ds)^m Y^n(t)]$ verify the same system of ODE's. Hence, since the Laplace transform of $Z(t) = \int_0^t Y(s)ds$ is known in closed form (see \cite{CIR}) and is finite in a radius around zero as noted by Dufrense~\cite{Dufresne}. Hence (according to section 30 in Billingsley~\cite{Billingsley}), the moments of $Z(t)$ determine its distribution and $\hat \Lambda_k(t) \to Z(t)$ as $k \to \infty$ in distribution. 
\end{proof}

%
%


\begin{example}(Gamma density)
Suppose that $dF(x) = f(x)dx$, where $f(x)$ is the PDF of a gamma distribution:
\begin{equation}\label{density:gamma}
f(x) = \frac{u^v}{\Gamma(v)}x^{v-1}\e^{-ux}, 
\end{equation}
where $u > 0$ and $v > 0$ are constants. By letting $v \downarrow 0$, the jumps become smaller and smaller, and if the parameters of the intensity process are adjusted simultaneously so that the intensity becomes higher and higher, the self-exciting process can be made into an infinite activity process. 

To that end, suppose that $\{v_k\}_{k \geq 1} $ is a sequence such that $v_k > 0$ for all $k \geq 1$, and $v_k \to 0$ as $k \to \infty$, and let $a_k := 1/\Gamma(v_k)$, for $k \geq 1$. Furthermore, suppose that for each $k \geq 1$, the stochastic intensity, $\lambda_k(t)$, defined in \eqref{lam_k}, has   
$$
\lambda_{0k} = c_0 \sqrt{\frac{\Gamma(v_k)(1+v_k)}{v_k}}, \ \ \beta_k = \sqrt{\frac{\Gamma(v_k)}{v_k(1+v_k)}}
$$
and that the drift rate is linear, 
$$
\mu_k(\lambda) = \left(\beta_k E[X_k] + c_1\right)\left(\lambda_{0k} - \lambda \right), 
$$
where $c_0,c_1 > 0$ are constants (and $X_k$ is gamma distributed with density \eqref{density:gamma} and $v = v_k$). Then, since the moments of the gamma distribution are given by 
$$
E[X_k^j] = \frac{v_k(v_k+1)\cdots (v_k+j-1)}{u^j},
$$
it follows that 
$$
\mu_k(\lambda) + \beta_kE[X_k]\lambda = (\beta_kE[X_k] + c_1)\lambda_{0k} - c_1\lambda,
$$
and clearly it holds that $a_k(\beta_kE[X_k] + c_1)\lambda_{0k} \to c_0/u$ as $k \to \infty$. Finally note that 
$$
\lim_{k \to \infty} a_k^{j-1}\beta_k^jE[X_k^j] = \begin{cases}
u^{-2} & \text{if } j=2 \\  
0 & \text{if } j > 2
\end{cases}
$$ 
So, according to Theorem \ref{thm:lam_convergence} we may conclude that the scaled intensity process tends to a process $Y(t)$ in distibution, where the dynamics of $Y(t)$ are given by the SDE
\begin{equation*}
dY(t) = (c_0 - \frac{c_1}{u}Y(t))dt + \frac{1}{u}\sqrt{Y(t)}dB(t),
\end{equation*} 
and $Y(0) = 0$. 
\end{example}

\section{A particular case: Linear intensity process}
\label{sec: linear}

We now consider the special case where the intensity process, $\lambda(t)$, is linear. That is,

\begin{equation}
\label{eq: linear_intensity}
d\lambda(t) = \alpha (\lambda_0 - \lambda(t))dt + \beta dU(t).
\end{equation}
In this section, we will study the expected value and variance of the intensity process in this special case. In this section we also assume that the jump-size distribution associated to each event is independent of the current value of the intensity process.

\subsection{The expected value of the intensity in the linear case}
\label{sec: expectation}

Let $m(t)$ denote the expected value of the intensity process, viewed as a function of time, so

\begin{equation}
m(t) := E[\lambda(t)].
\end{equation}
Define $\rho:= \beta E[X] - \alpha$. Eyjolfsson and Tj{\o}stheim \cite{EyjolfssonTjostheim} (see their equation (13)), use Dynkin's formula and Fubini's theorem to derive that

\begin{equation}
\label{eq: ode}
\begin{array}{lll}
m'(t) &=& \alpha \lambda_0 + \rho m(t) \\[\smallskipamount]
m(0) &:=& \lambda.
\end{array}
\end{equation}
Equation \eqref{eq: ode} is an ordinary differential equation (ODE) with solution

\begin{equation}
\label{eq: explicit_ode_expectation}
m(t) = m_0 + m_1e^{\rho t},
\end{equation}
where $m_0= -\alpha \lambda_0\rho^{-1}$ and $m_1 = \alpha \lambda_0\rho^{-1} + \lambda$. See Eyjolfsson and Tj{\o}stheim \cite{EyjolfssonTjostheim} for the details of this derivation, or the following Section \ref{sec: variance}, for how to apply Dynkin's formula to obtain the ODE. 

Based on the ordinary differential equation \eqref{eq: ode}, we separate the long term behaviour of the expected value, $m(t)$, into three cases:

\begin{itemize}
\item{If $\rho > 0$: In this case, $E[\lambda(t)]$ grows exponentially with time.}

\item{If $\rho = 0$: In this case, from \eqref{eq: ode} we get $m(t) = \lambda +\alpha \lambda_0 t$, so $E[\lambda(t)]$ grows linearly with time.}

\item{If $\rho < 0$: In this case, $E[\lambda(t)]$ is bounded and $E[\lambda(t)] \to -\alpha \lambda_0\rho^{-1}$ as $t \to \infty$.}
\end{itemize}

Recall that $\rho:= \beta E[X] - \alpha$. Hence, $\rho > 0$ means that $E[X]> \frac{\alpha}{\beta}$, where $\alpha$ is the drift term and $\beta$ is the diffusion term for the intensity SDE. Hence, the interpretation is that if the jump sizes are independent and identically distributed, and the expected jump size is larger than the fraction of the drift term over the diffusion term, then the expected intensity rate will grow exponentially with time. Similarly, $\rho=0$ means that $E[X]= \frac{\alpha}{\beta}$. Hence, if the expected jump size is in perfect balance with the fraction of the drift term over the diffusion term, then the expected intensity will grow linearly with time. Finally, $\rho < 0$ means that $E[X]< \frac{\alpha}{\beta}$. So, if the expected jump size is smaller than the fraction of the drift term over the diffusion term, then the expected intensity is bounded (and we know what it converges to).


\subsection{The variance of the intensity in the linear case}
\label{sec: variance}

Let $v(t)$ denote the second order moment of the intensity process, viewed as a function of time, so
\begin{equation}
v(t) := E[\lambda^2(t)].
\end{equation}
To determine the second moment, we use the same idea as in Section \ref{sec: expectation}, and as in Eyjolfsson and Tj{\o}stheim \cite{EyjolfssonTjostheim}: We use Dynkin's formula and Fubini's theorem to derive an ordinary differential equation for $v(t)$. From Dynkin's formula with $f(x)=x^2$,
\begin{equation}
\label{eq: dynkin}
E[\lambda^2(t)] = \lambda^2 + E[\int_0^t \mathcal{A} f(\lambda(r))dr]
\end{equation}
\noindent where $\mathcal{A}(\cdot)$ is the infinitesimal (or extended) generator of $\lambda(t)$. 
Note that according to \eqref{def:moment} it holds that

\[
\begin{array}{llll}
(\mathcal{A}f)(\lambda) &=& 2\lambda \alpha (\lambda_0 - \lambda) + \lambda \int ((\lambda + \beta x)^2 - \lambda^2) \nu(\lambda, dx) \\[\smallskipamount]
&=& \lambda (2 \alpha \lambda_0 + \beta^2 E[X^2]) + \lambda^2 (2 \beta E[X] - 2\alpha) \\[\smallskipamount]
&=& A \lambda + 2\rho \lambda^2,
\end{array}
\]
\noindent where we define $A := 2 \alpha \lambda_0 + \beta^2 E[X^2]$ and $\rho = \beta E[X] - \alpha$ as before.
By inserting this into the application of Dynkin's formula above in equation \eqref{eq: dynkin}, and using Fubini's theorem to change the order of integration and expectation, we find
\[
\begin{array}{llll}
v(t) &=& \lambda^2 + \int_0^t E[A \lambda(r) + 2\rho \lambda(r)^2]dr \\[\smallskipamount]
&=& \lambda^2 + \int_0^t (Am(r) + 2\rho v(r)) dr.
\end{array}
\]
Hence, 
\begin{equation}
\label{eq: ode_variance}
\begin{array}{llll}
v'(t) &=&  Am(t) + 2\rho v(t) \\[\smallskipamount]
v(0) &=& \lambda^2.
\end{array}
\end{equation}
Recall that we have an explicit expression for $m(t)$ from equation \eqref{eq: explicit_ode_expectation}. Hence, equation \eqref{eq: ode_variance} is an ordinary differential equation in $v(t)$ which can be solved by standard techniques by inserting the expression for $m(t)$ from \eqref{eq: explicit_ode_expectation}.

We now consider the same three cases as in Section \ref{sec: expectation}
\begin{itemize}
\item{If $\rho > 0$: By observing the ODE \eqref{eq: ode_variance}, we see (as in Section \ref{sec: expectation}) that the intensity process may explode in finite time.}
\item{If $\rho = 0$: In this case, 
$
v'(t) = A \lambda, 
$
so $v(t) = \lambda^2 e^{\frac{A}{\lambda}t}$. Hence,
\[
\begin{array}{llll}
Var(\lambda(t)) &=& E[\lambda(t)^2]- (E[\lambda(t)])^2 \\[\smallskipamount]
&=& v(t) - m(t)^2 \\[\smallskipamount]
&=& \lambda^2 (e^{\frac{A}{\lambda}t} -1).
\end{array}
\]
So when $t \rightarrow \infty$, $Var(\lambda(t)) \rightarrow \infty$ as well, since $A > 0$ and $\lambda > 0$.
}
\item{If $\rho < 0$: In this case, we see from the expression for $m(t)$ in equation \eqref{eq: explicit_ode_expectation} that $m(t) \rightarrow -\alpha \lambda_0\rho^{-1}$ (a positive constant) as $t \rightarrow \infty$. So, since $\rho < 0$, we see from the ODE \eqref{eq: ode_variance} that $v'(t) < 0$ for a sufficiently large $v(t)$. That is, for a second order moment, the derivative of this moment becomes negative. This will stabilise the intensity process. Hence, for $\rho < 0$, the intensity process is stable and will not explode in finite time. Note that in the case where $v(t)$ is small, the process is already stable, so the sign of the derivative $v'(t)$ is not important.
}
\end{itemize}

The interpretations of these items are similar to those of Section \ref{sec: expectation}: Since $\rho:= \beta E[X] - \alpha$, $\rho > 0$ means that $E[X]> \frac{\alpha}{\beta}$, where $\alpha$ is the drift term and $\beta$ is the diffusion term for the intensity SDE. Hence, the interpretation is that if the jump sizes are independent and identically distributed, and the expected jump size is larger than the fraction of the drift term over the diffusion term, then the intensity process may explode in finite time. Similarly, $\rho=0$ means that $E[X]= \frac{\alpha}{\beta}$. Hence, even when the expected jump size is in perfect balance with the fraction of the drift term over the diffusion term, the intensity process can explode as time goes to infinity. However, it will not explode in finite time. Finally, $\rho < 0$ means that $E[X]< \frac{\alpha}{\beta}$. So, if the expected jump size is smaller than the fraction of the drift term over the diffusion term, the intensity process is stable, in the sense that it will not explode in finite time.
%
%
%
\subsection{The second moment of the integrated intensity in the linear case}\label{sec:cov}
The second moment differential equation \eqref{eq: ode_variance} has the solution
\begin{align}\label{eq:vt}
v(t) = v_0 + v_1\e^{\rho t} + v_2\e^{2\rho t},
\end{align}
where $v_0 = A\alpha\lambda_0(2\rho^2)^{-1}$, $v_1 = -A\rho^{-1}(\alpha\lambda_0\rho^{-1} + \lambda) $, and $v_2 = A(2\rho^2)^{-1}(\alpha\lambda_0 + 2\rho\lambda) + \lambda^2$ are constants. We note that the preceding observations are consistent with the properties of the solution that were observed in Section \ref{sec: variance}. This solution can in turn be used to determine the second moment of the integrated intensity, $\Lambda(t)$. Note that by an application of Fubini it holds that
$$
E[\Lambda^2(t)] = \int_0^t\int_0^t E[\lambda(r)\lambda(s)] drds.
$$
Suppose that $s \geq r$. Then, it holds that 
\begin{align*}
E[\lambda(r)\lambda(s)] = v(r) + E[\lambda(r)(\lambda(s) - \lambda(r)],
\end{align*}
where $v(t) = E[\lambda^2(t)]$. From what we know about the first moment of a linear intensity, given an initial value, 
\[
\begin{array}{lllll}
E[\lambda(r)(\lambda(s) - \lambda(r)] &=& E[\lambda(r)E[(\lambda(s) - \lambda(r)|\lambda(r)]] \\[\smallskipamount]
&=& E[\lambda(r)\left((\frac{\alpha \lambda_0}{\rho} + \lambda(r))\e^{\rho(s-r)} -\frac{\alpha \lambda_0}{\rho} - \lambda(r)\right)] \\[\smallskipamount]
&=& \left(\frac{\alpha \lambda_0}{\rho}m(r) + v(r)\right)\e^{\rho(s-r)} - \left( \frac{\alpha \lambda_0}{\rho}m(r) + v(r)\right).
\end{array}
\]
To simplify notation, suppose that $\phi(r) := \frac{\alpha \lambda_0}{\rho}m(r) + v(r)$, for any $r > 0$, then it follows that
\begin{align*}
E[\Lambda^2(t)] &= \int_0^t\left(\int_0^s (v(r) + \phi(r)(\e^{\rho(s-r)}-1))dr + \int_s^t (v(s) + \phi(s)(\e^{\rho(r-s)}-1))dr \right)ds \\
&= 2\int_0^t \int_s^t (v(s) + \phi(s)(\e^{\rho(r-s)}-1))dr ds \\
&= 2\int_0^t \left((t-s)(v(s) - \phi(s)) + \frac{\e^{\rho(t-s)}-1}{\rho}\phi(s)\right)ds 
\end{align*}
By writing $m(s) = m_0 + m_1\e^{\rho s}$, like we do in \eqref{eq: explicit_ode_expectation}, it holds that  
\begin{align*}
\int_0^t (t-s)(v(s) - \phi(s))ds &= m_0 \int_0^t (t-s)(m_0 + m_1\e^{\rho s}) ds \\
&= m_0\left(\frac{m_0}{2}t^2 + m_1\frac{\e^{\rho t}-1-\rho t}{\rho^2} \right).
\end{align*}
Similarly, there exist constants $c_0$, $c_1$ and $c_2$ such that $\phi(s) = c_0 + c_1\e^{\rho s} + c_2\e^{2\rho s}$, so
\[
\begin{array}{llll}
\int_0^t \frac{\e^{\rho(t-s)}-1}{\rho}\phi(s)ds &=& \int_0^t \frac{\e^{\rho(t-s)}-1}{\rho} \left( c_0 + c_1\e^{\rho s} + c_2\e^{2\rho s}\right)ds \\[\smallskipamount]
&=& \frac{1}{\rho}\int_0^t \left(c_0 \e^{\rho t}\e^{-\rho s} + (c_1\e^{t\rho} - c_0) + (c_2\e^{\rho t} - c_1)\e^{\rho s} - c_2\e^{2\rho s}\right)ds \\[\smallskipamount]
&=&\frac{1}{\rho} \left(c_0\frac{\e^{\rho t}-1}{\rho} + (c_1\e^{t\rho} - c_0)t + (c_2\e^{\rho t} - c_1)\frac{\e^{\rho t} - 1}{\rho} -c_2 \frac{\e^{2\rho t} - 1}{2\rho}\right).
\end{array}
\]

Therefore, using that $c_0=v_0-m_0^2$, $c_1=v_1-m_0m_1$ and $c_2=v_2$ where $m_0, m_1, v_0, v_1, v_2$ are the constants in the first and second moment functions, \eqref{eq: explicit_ode_expectation} and \eqref{eq:vt}, respectively, we may conclude that 
\begin{align*}
E[\Lambda^2(t)] = k_0 + k_1t + k_2t^2 + (C_0+C_{1}t)e^{\rho t} + C_2e^{2\rho t},
\end{align*}
where the constants $k_0, k_1, k_2, C_0, C_1, C_2$ are given by
\[
\begin{array}{lllll}
k_0 &=& \frac{2m_0(m_0 - 2m_1) - 2v_0 + 2v_1 + v_2}{\rho^2}
\\[\medskipamount]

k_1 &=& \frac{2m_0(m_0-\rho^2 m_1) - 2v_0}{\rho} \\[\medskipamount]

k_2 &=& m_0^2 \\[\medskipamount]

C_0 &=& \frac{2(m_0(2m_1-m_0) + v_0 - v_1 - v_2)}{\rho^2} \\[\medskipamount]

C_1 &=& \frac{2(v_1-m_0m_1)}{\rho} \\[\medskipamount]
C_2 &=& \frac{v_2}{\rho^2},
\end{array}
\]
and $m_0, m_1, v_0, v_1, v_2$ are the constants in the first and second moment functions, \eqref{eq: explicit_ode_expectation} and \eqref{eq:vt}, respectively.
It follows that if $\rho < 0$ is close to zero, then the effects of a jump fade out slower, than if the $\rho < 0$ is further away from zero.
\subsection{Convergence to deterministic intensity}

In this subsection, we will study what happens to the intensity process $\lambda(t)$ if $\beta>0$ and $\alpha>0$ both converge towards zero while $\rho$ is kept constant. Note that according to the definition of $\rho$
\[
\begin{array}{llll}
d \lambda(t) &=& -\alpha \lambda(t)dt + \alpha \lambda_0 dt + \beta dU(t) \\[\smallskipamount]
&=& \rho \lambda(t)dt + \alpha \lambda_0 dt + \beta d(U(t) - E[X]t) \rightarrow \rho \lambda(t)dt
\end{array}
\]
as $\alpha, \beta \rightarrow 0$ while $\rho$ is kept constant.
This means that the stochastic differential equation which determined the intensity process converges towards an ordinary (deterministic) differential equation as $\alpha, \beta \rightarrow 0$ while $\rho$ is kept constant. This ODE is
$
\lambda' = \rho \lambda, 
$
which means that $\lambda(t) = \lambda e^{\rho t}$. From this, it follows that the self-exciting process $U(t)$ converges to a non-homogeneous Poisson process with intensity process $\lambda(t) = \lambda e^{\rho t}$. Hence, if $\rho < 0$, then the intensity converges to zero, and no more jumps occur. If $\rho=0$, then the intensity converges to a constant $\lambda(t) = \lambda$, and if $\rho > 0$, then the intensity tends to infinity as $t \to \infty$.

\section{Conclusions and future work}\label{sec:Conclusion}

To conclude, the purpose of this paper has been to investigate properties of self-exciting processes with intensity processes given by an SDE. The following are the main contributions of the paper. We have:

\begin{itemize}
\item{Proved that the scaling limit of the intensity process equals the square root process in distribution. We have also proved a similar result for the integrated intensity process.}
\item{Derived explicit expressions for the expectation and variance of the intensity in the case of an intensity given by a linear SDE.}
\item{Shown that the intensity process may explode in finite time, or be stable (in the sense that it does not explode in finite time). Hence, SDE driven self-exciting stochastic processes may have both finite and infinite activity.}
\item{Proved that the intensity process converges to a deterministic intensity as the drift and diffusion coefficients go to zero, as long at the expected jump size equals the fraction of the drift and diffusion terms.}
\end{itemize} 

There is still much to be done in investigating SDE driven self-exciting processes. Work in progress is finding the moments of $U(t)$ and deriving scaling limit results. Furthermore, applications, for instance looking into stochastic optimal control problems where the state process is given by an SDE driven self-exciting process would be interesting. This is left for future research.

\section*{Funding}
This work was supported by the Research Council of Norway under the SCROLLER project, project number 299897.

\end{document}